\documentclass{amsart}
\usepackage{latexsym, graphicx, amsmath, amssymb, mathtools, cite, hyperref}
\usepackage{tikz}

\newtheorem{theorem}{Theorem}[section]
\newtheorem{lemma}[theorem]{Lemma}

\theoremstyle{definition}
\newtheorem{definition}[theorem]{Definition}

\theoremstyle{remark}
\newtheorem{remark}[theorem]{Remark}

\numberwithin{equation}{section}

\newcommand{\sm}{\setminus}

\renewcommand{\phi}{\varphi}
\renewcommand{\l}{\ell}

\title{VC-dimension of hyperplanes over finite fields}
\author[\scalebox{0.67}{Ascoli, Betti, Cheigh, Iosevich,  Jeong, Liu, McDonald, Milgrim, Miller, Romero Acosta, Velazquez Iannuzzelli}]{Ruben Ascoli}
\author[]{Livia Betti}
\author[]{Justin Cheigh}
\author[]{Alex Iosevich}
\author[]{Ryan Jeong}
\author[]{Xuyan Liu}
\author[]{Brian McDonald}
\author[]{Wyatt Milgrim}
\author[]{Steven J. Miller}
\author[]{Francisco Romero Acosta}
\author[]{Santiago Velazquez Iannuzzelli}

\begin{document}

%\begin{abstract} 
%\end{abstract}

\begin{abstract}
Let $\mathbb{F}_q^d$ be the $d$-dimensional vector space over the finite field with $q$ elements.  For a subset $E\subseteq \mathbb{F}_q^d$ and a fixed nonzero $t\in \mathbb{F}_q$, let $\mathcal{H}_t(E)=\{h_y: y\in E\}$, where $h_y$ is the indicator function of the set $\{x\in E: x\cdot y=t\}$.  Two of the authors, with Maxwell Sun, showed in the case $d=3$ that if $|E|\geq Cq^{\frac{11}{4}}$ and $q$ is sufficiently large, then the VC-dimension of $\mathcal{H}_t(E)$ is 3.  In this paper, we generalize the result to arbitrary dimension and improve the exponent in the case $d=3$.

\end{abstract}

\maketitle

\section{Introduction}

Vapnik and Chervonenkis \cite{VC71} introducted the VC-dimension in 1971 in the context of learning theory.  For an introduction to the subject, see for example \cite{DS14}.  Given a domain $X$ and a collection $\mathcal{H}$ of functions $h:X\to \{0,1\}$, consider the learning task of trying to identify an unknown element $f\in \mathcal{H}$ by sampling finitely many points $x_1,...,x_m\in X$ from an unknown probability distribution $D$, and recording the values $f(x_1),...,f(x_m)$.  One desires an algorithm which takes this input and produces a hypothesis $h\in H$ which with high probability has small error with respect to $f$.  To make this precise, we introduce some definitions.

\begin{definition}
Given a set $X$, a probability distribution $D$, and a labeling function $f:X\to \{0,1\}$, let $h$ be a hypothesis, i.e., $h:X\to \{0,1\}$, and define
$$
L_{D,f}(h)=\mathbb{P}_{x\sim D}[h(x)\neq f(x)],
$$
where $\mathbb{P}_{x\sim D}$ means that $x$ is being sampled according to the probability distribution $D$.
\end{definition}

\begin{definition}
A hypothesis class $\mathcal{H}$ is PAC (probably approximately correct) learnable if there exists a function 
$$
m_{\mathcal{H}}:(0,1)^2\to \mathbb{N}
$$
and a learning algorithm with the following property:  For every $\epsilon,\delta\in (0,1)$, for every distribution $D$ over $X$, and for every labeling function $f:X\to \{0,1\}$, if there is some hypothesis $h\in \mathcal{H}$ such that $L_{D,f}(h)=0$, then when running the learning algorithm on $m\geq m_{\mathcal{H}}(\epsilon,\delta)$ i.i.d. examples generated by $D$, and labeled by $f$, the algorithm returns a hypothesis $h$ such that, with probability at least $1-\delta$ (over the choice of $(x_1,...,x_m)\sim D^m$),
$$
L_{D,f}(h)\leq \epsilon.
$$
\end{definition}
The VC-dimension characterizes PAC learnability, in light of the fundamental theorem of statistical learning; $\mathcal{H}$ is PAC learnable if and only if the VC-dimension is finite.  Moreover, there are quantitative bounds for $m_{\mathcal{H}}(\epsilon,\delta)$ based on VCdim$(\mathcal{H})$, with smaller VC-dimension allowing smaller effective sample sizes.  In order to define the VC-dimension, we must first define shattering.

\begin{definition} \label{shatteringdef} Let $X$ be a set and ${\mathcal H}$ a collection of functions from $X$ to $\{0,1\}$. We say that ${\mathcal H}$ shatters a finite set $C \subset X$ if the restriction of ${\mathcal H}$ to $C$ yields every possible function from $C$ to $\{0,1\}$. \end{definition} 

\vskip.125in 

\begin{definition} \label{vcdimdef} Let $X$ and ${\mathcal H}$ be as above. We say that a non-negative integer $d$ is the VC-dimension of ${\mathcal H}$ if there exists a set $C \subset X$ of size $n$ that is shattered by ${\mathcal H}$, and no subset of $X$ of size $n+1$ is shattered by ${\mathcal H}$. \end{definition} 

\vskip.125in

Recent work has connected the VC-dimension to point configuration problems over finite fields.  For $x\in \mathbb{F}_q^d$, let
$$
||x||=x_1^2+\cdots +x_d^2.
$$
For a subset $E\subseteq \mathbb{F}_q^d$, and a fixed nonzero $t\in \mathbb{F}_q$, let
$$
\mathcal{H}_t^{dist}(E):=\{f_y: y\in E\},
$$
where $f_y(x)=1$ if and only if $||x-y||=t$.  Fitzpatrick, Iosevich, McDonald, and Wyman \cite{IMW23} showed in the case $d=2$ that if $|E|\geq Cq^{\frac{15}{8}}$, $q$ sufficiently large, then VCdim$(\mathcal{H})=3$.  The exponent $\frac{15}{8}$ was recently improved to $\frac{13}{7}$ by Thang Pham \cite{P23}, refining the method of \cite{IMW23}. 
 In the case when $E=\mathbb{F}_q^2$ this is trivial, and one may see by induction that in general
$$
\text{VCdim}(\mathcal{H}_t^{dist}(\mathbb{F}_q^d))=d+1.  
$$
In dimensions $d\geq 3$, it is still an open problem whether one can find a threshold $\alpha\in (0,d)$ so that whenever $E\geq C_dq^{\alpha}$ for some constant $C_d$ independent of $q$,  
$$
\text{VCdim}(\mathcal{H}_t^{dist}(E))=d+1.
$$
The best partial result in arbitrary dimension is a corollary of the main theorem from the previous result by the authors of this paper \cite{SMALL22}.  In the previous paper, we considered a related hypothesis class with two parameters.  Let
$$
\mathcal{H}_t^{\ast}(E)\coloneqq\{h_{u,v}: u,v\in E\},
$$
where $h_{u,v}(x)=1$ if and only if $||x-u||=||x-v||=t$.  In \cite{SMALL22}, we showed that whenever
$$
|E|\geq \left\{\begin{array}{ll}
Cq^{\frac{7}{4}} & d=2 \\
Cq^{\frac{7}{3}} & d=3 \\
Cq^{d-\frac{1}{d-1}} & d\geq 4
\end{array}\right.
$$
and $q$ is sufficiently large, the VC-dimension of $\mathcal{H}_t^{\ast}(E)$ is equal to $d$.  As we explain in section 5 of that paper, it follows that with the same restriction on the size of $E\subseteq \mathbb{F}_q^d$, the VC-dimension of $\mathcal{H}_t^{dist}(E)$ is either $d$ or $d+1$.  
\\
\\
These hypothesis classes are closely related to the setting for our main result in this paper.  For a subset $E\subseteq \mathbb{F}_q^d$, and fixed nonzero $t\in \mathbb{F}_q^d$, consider the hypothesis class
$$
\mathcal{H}_t(E):=\{h_y: y\in E\},
$$
where $h_y(x)=1$ if and only if $x\cdot y=t$.  Iosevich, McDonald, and Sun \cite{IMS23} studied this in the case $d=3$, and showed that when $|E|\geq Cq^{\frac{11}{4}}$, the VC-dimension of $\mathcal{H}_t(E)$ is 3.  Note that the VC-dimension of $\mathcal{H}_t(\mathbb{F}_q^d)$ is $d$ and not $d+1$, since a hyperplane in $\mathbb{F}_q^d$ is determined by $d$ points, whereas a sphere in $\mathbb{F}_q^d$ is determined by $d+1$ points.  This result required a different approach from that of \cite{IMW23}, since the latter used the fact that the property $||x-y||=t$ is translation invariant, whereas the property $x\cdot y=t$ is not.  
\\
\\
In this paper, we prove the following theorem which generalizes \cite{IMS23} to arbitrary dimension, and improves the exponent from $\frac{11}{4}$ to $\frac{5}{2}$ in the case $d=3$.  
\begin{theorem}\label{main}
For $d\geq 3$, if $|E|\geq C_dq^{d-\frac{1}{d-1}}$ for an appropriate constant $C_d$ depending only on $d$, and for $q$ sufficiently large, then the VC-dimension of $\mathcal{H}_t(E)$ is equal to $d$.
\end{theorem}
We will prove this result with techniques similar to \cite{SMALL22}, which will be strong enough to compute the VC-dimension of a large subset $E\subseteq \mathbb{F}_q^d$ in arbitrary dimension, in contrast to the situation for distances.
\\
\\
The results discussed above can be expressed in terms of graph embeddings $\phi:G\hookrightarrow \mathcal{G}_t(E)$ for appropriate graphs $G$, where $\mathcal{G}_t(E)$ is the distance (resp. dot product) graph, i.e., the vertices are points in $E$, with an edge $x\sim y$ whenever $||x-y||=t$ (resp. $x\cdot y=t$). For relevant results on graph embeddings in the distance and dot product graphs, see for example \cite{BCCHIP16, CIKR10, IJM21, IMS23, IMW23, IP19, IR07}.

\section{Proof of main theorem}

Consider a large subset $E\subseteq \mathbb{F}_q^d$, and a fixed nonzero $t\in \mathbb{F}_q$.  We will use theorem 2.1 from \cite{CIKR10}, which counts pairs $(x,y)\in E^2$ with $x\cdot y=t$.  
\begin{theorem}[\hspace{1sp}\cite{CIKR10}]\label{edge}
For non-negative functions $f,g:\mathbb{F}_q^d\to \mathbb{R}$, 
$$
\sum_{x\cdot y=t}f(x)g(y)=q^{-1}||f||_{L^1}||g||_{L^1}+R(t),
$$
where
$$
|R(t)|\leq ||f||_{L^2}||g||_{L^2}q^{\frac{d-1}{2}}.
$$
\end{theorem}
In particular, when $f,g$ are both chosen to be the indicator function of $E$, we see that
$$
|\{(x,y)\in E^2: x\cdot y =t\}|=\frac{|E|^2}{q}+O\left(q^{\frac{d-1}{2}}|E|\right),
$$
and the error term is much smaller than the main term when $|E|=\omega\left(q^{\frac{d+1}{2}}\right)$.  We use this fact, along with Holder's inequality, to count the number of $k$-stars in the dot-product graph on $E$.  
\begin{definition}\label{star}
A (k+1)-tuple $(y,x_1,...,x_k)$ of points in $\mathbb{F}_q^d$ is a $k$-star if $y\cdot x_i=t$ for each $i=1,...,k$.  If all the $x_i$ are distinct, we say $(y,x_1,...,x_k)$ is a non-degenerate $k$-star.
\end{definition}

\begin{figure}
\centering
\includegraphics[width=0.6\textwidth]{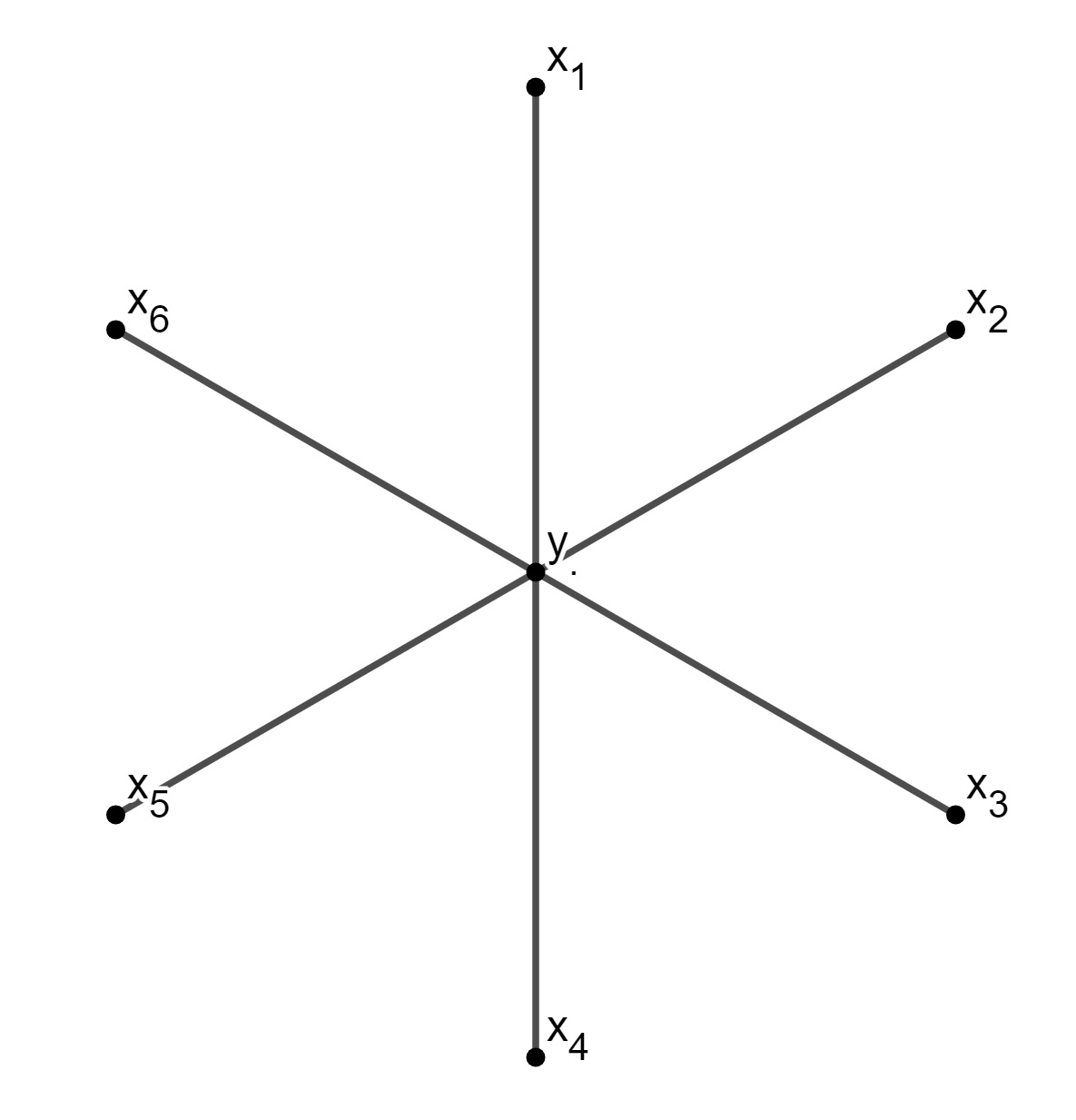}
\caption{A 6-star realized as a subgraph of the dot product graph $\mathcal{G}_t(E)$.}
\label{Star}
\end{figure}

\begin{definition}
Let $\mathcal{G}_t(E)$ be the dot product $t$ graph on $E$, i.e., the graph with vertex set $E$ and an edge $x\sim y$ whenever $x\cdot y=t$.
\end{definition}
\begin{lemma}\label{star counting}
Let 
$$
N_k(E):=|\{(y,x_1,...,x_k)\in E^{k+1}: x_i \ \mathrm{distinct}, \ y\cdot x_i=t \ \forall i\}|
$$
be the number of non-degenerate $k$-stars in $\mathcal{G}_t(E)$.  If $|E|\geq C_kq^{\frac{d+1}{2}}$ for an appropriate constant $C_k$ depending only on $k$, then
$$
N_k(E)\geq \frac{|E|^{k+1}}{2q^k}.
$$
\end{lemma}
\begin{proof}
For $x\in E$, let 
$$
\psi(x)=\sum_{\substack{y\in E \\ x\cdot y =t}}1
$$ 
be the number of neighbors of $x$ in $\mathcal{G}_t(E)$.  Then
$$
N_k(E):=\sum_{x\in E}\psi(x)(\psi(x)-1)\cdots (\psi(x)-k+1)
$$
$$
\geq \sum_{x\in E}\phi(x)^k,
$$
where $\phi(x)=\max(\psi(x)-k+1,0)$.  By Holder's inequality,
$$
\left(\sum_{x\in E}\phi(x)\right)^k\leq \left(\sum_{x\in E}{\phi(x)^k}\right)\left(\sum_{x\in E}1\right)^{k-1}
\leq |E|^{k-1}N_k(E).
$$
To get the desired lower bound for $N_k(E)$, it suffices to bound $\sum_{x\in E}{\phi(x)}$ from below.  We obtain such a lower bound as a result of Theorem \ref{edge}:
\begin{align*}
\sum_{x\in E}\phi(x)&\geq 
\sum_{x\in E}(\psi(x)-k+1) \\
&=\frac{|E|^2}{q}+O\left(q^{\frac{d-1}{2}}|E|\right)-(k-1)|E|
\geq 2^{-\frac{1}{k}}\frac{|E|^2}{q},
\end{align*}
assuming $|E|\geq C_kq^{\frac{d+1}{2}}$ for an appropriate constant $C_k$ depending only on $k$.  This yields
$$
N_k(E)\geq \frac{|E|^{k+1}}{2q^k}.
$$
\end{proof}
Having obtained a lower bound for the number of $k$-stars in $\mathcal{G}_t(E)$, we are particularly interested in the case $k=d$, and particularly those stars $(y,x_1,...,x_d)$ with the property that $\{x_1,...,x_d\}\subseteq \mathbb{F}_q^d$ is a linearly independent set of vectors.  Therefore, we would like to find an upper bound for the number of $d$-stars $(y,x_1,...,x_d)$ formed from linearly dependent sets $\{x_1,...,x_d\}$.  
\begin{lemma}\label{independent}
Let $\mathcal{N}_d(E)$ be the number of $d$-stars $(y,x_1,...,x_d)$ in $\mathcal{G}_t(E)$ such that $\{x_1,...,x_d\}$ is a linearly independent set.  If 
$$
|E|\geq C_dq^{d-\frac{1}{d-1}},
$$
for $q$ sufficiently large, then 
$$
\mathcal{N}_d(E)\geq \frac{|E|^{d+1}}{3q^d}.
$$
\end{lemma}
\begin{proof}
In a star $(y,x_1,...,x_d)$, if $\{x_1,...,x_d\}$ is linearly dependent, we may assume without loss of generality that 
$$
x_d\in \text{Span}(x_1,...,x_{d-1}).
$$
For a given $y\in E$, there are $\psi(y)$ points $x\in E$ such that $x\cdot y=t$.  Therefore, there are at most $\psi(y)^{d-1}$ choices for the first $d-1$ points $x_1,...,x_{d-1}$.  Once $y,x_1,...,x_{d-1}$ are fixed, we see that the point $x_d$ lies on the hyperplane $\{x\in E: x\cdot y=t\}$ as well as the hyperplane $\text{Span}(x_1,...,x_{d-1})$.  These are not the same hyperplane, as only one of them contains the origin since $t\neq 0$.  Moreover, their intersection is nonempty since it contains $x_d$, and so we conclude that $x_d$ must be chosen from a $(d-2)$-dimensional subspace, which must have $q^{d-2}$ points.  Putting this together, we find that the number of stars $(y,x_1,...,x_d)$ in $\mathcal{G}_t(E)$ with the set $\{x_1,...,x_d\}$ being linearly dependent is bounded by
$$
dq^{d-2}\sum_{y\in E}\psi(y)^{d-1} 
\leq dq^{d-2}q^{(d-1)(d-2)}\sum_{y\in E}\psi(y)
$$
$$
\lesssim dq^{d(d-2)}\frac{|E|^2}{q},
$$
since $\phi(y)\leq q^{d-1}$ for any $y$. The factor of $d$ comes from the fact that we chose $x_d\in \text{Span}(x_1,...,x_{d-1})$.  The last line follows from Theorem \ref{edge}. We find that
$$
dq^{d(d-2)}\frac{|E|^2}{q}
< \frac{|E|^{d+1}}{6q^d}
$$
as long as $|E|\geq C_dq^{d-\frac{1}{d-1}}$ for an appropriate constant $C_d$. Finally, Lemma \ref{star counting} finishes the proof of the statement.
\end{proof}
\begin{definition}
For a $d$-star $\mathcal{S}=(y,x_1,...,x_d)$, we call $L=\{x_1,...,x_d\}$ the leaf set.  We say a subset $A=\{x_{n_1},...,x_{n_k}\}$ of the leaf set is bad with respect to $\mathcal S$ if for every $z\in E$ satisfying $z\cdot x_{n_i}=t$ for all $i=1,...,k$, there is some $x\in L\sm A$ with $z\cdot x=t$ as well.  
\end{definition}
\begin{remark}
Our definition of a bad set is designed for testing whether the set $\{x_1,...,x_d\}$ is shattered by $\mathcal{H}_t(E)$.  In particular, it follows immediately from definitions that $\{x_1,...,x_d\}\subseteq E$ is shattered if and only if there is some $y\in E$ so that $\mathcal{S}=(y,x_1,...,x_d)$ is a $d$-star in $\mathcal{G}_t(E)$, and $\{x_1,...,x_d\}$ admits no bad subset of size $k=1,...,d-1$.
\end{remark}
With this in mind, our strategy is to show that a generic $d$-star in $\mathcal{G}_t(E)$ with a linearly independent leaf set admits no bad sets.  To see this, we bound the number of $d$-stars corresponding to a given bad set. 
\begin{definition}
Given a set $B=\{b_1,...,b_k\}$ which is bad in some $d$-star $\mathcal{S}=(y,x_1,...,x_d)$ with linearly independent leaf set $L=\{x_1,...,x_d\}$, let 
$$
\mathcal{Q}(B)\coloneqq\{x\in E: x\cdot b_i=t \ \forall i=1,...,k\}
$$
\end{definition}
If $\mathcal{Q}(B)$ is small, this restricts the number of choices for the point $y$ in a star $\mathcal{S}=(y,x_1,...,x_d)$ containing $B$.  If $\mathcal{Q}(B)$ is large, on the other hand, we will see that this restricts the number of choices for the leaf set.  The following lemma will allow us to separate into cases based on the size of $\mathcal{Q}(B)$.
\begin{lemma}\label{sizelemma}
Suppose that $B=\{b_1,...,b_k\}$ is bad in some star $\mathcal{S}=(y,x_1,...,x_d)$, and that
$$
|\mathcal{Q}(B)|>q^{r-1}.
$$
Then for any $y\in \mathcal{Q}(B)$, there is a subset $J\subseteq \mathcal{Q}(B)$ of size $r$, not containing $y$, so that $\{y\}\cup J$ is linearly independent. 
\end{lemma}
\begin{proof}
Fix $b\in B$, so that every point $x\in \mathcal{Q}(B)$ lies on the hyperplane $H_b$ defined by $x\cdot b=t$.  Suppose that $J$ is the largest subset of $\mathcal{Q}(B)$, with the desired property that $\{y\}\cup J$ is linearly independent and $J$ does not contain $y$.  For any 
$$
z\in \mathcal{Q}(B)\setminus \text{Span}(\{y\}\cup J),
$$
we see that $\{y,z\}\cup J$ is linearly independent.  Since we assumed that $J$ is maximal, this means that
$$
\mathcal{Q}(B)\setminus \text{Span}(\{y\}\cup J)=\emptyset.
$$
Therefore,
$$
\mathcal{Q}(B)=\mathcal{Q}(B)\cap \text{Span}(\{y\}\cup J)\subseteq H_b\cap \text{Span}(\{y\}\cup J).
$$
Also note that $H_b$ does not contain $\text{Span}(\{y\}\cup J)$ since the former does not contain 0, while the latter does.  Thus, their intersection is an affine subspace of dimension at most $|J|$, having at most $q^{|J|}$ points.  Therefore,
$$
q^{r-1}<|\mathcal{Q}(B)|\leq |H_b\cap \text{Span}(\{y\}\cup J)|\leq q^{|J|},
$$
so $|J|\geq r$.
\end{proof}
\begin{lemma}\label{bad}
For $E\subseteq \mathbb{F}_q^d$, the number of $d$-stars in $E$ with linearly independent leaf set containing a bad set of size $k$ is at most 
$$
C_d'|E|^kq^{d^2-kd-d+k},
$$
for an appropriate constant $C_d'$.
\end{lemma}
\begin{proof}
We fix a linearly independent set $B=\{x_1,...,x_k\}\subseteq E$, $1\leq k \leq d-1$, and count the ways to extend this to a $d$-star $(y,x_1,...,x_d)$ for which $B$ is a bad set and $\{x_1,...,x_d\}$ is linearly independent.  Note that permuting the elements of $\{x_1,...,x_d\}$ does not change any of this data, so up to a constant depending only on $d$, this is the only case we need to consider.  We assume $B$ is bad in at least one $d$-star, $\mathcal{S}_0=(y^0,x_1,...,x_k,x_{k+1}^0,...,x_d^0)$, with $\{x_1,...,x_k,x_{k+1}^0,...,x_d^0\}$ linearly independent, since otherwise the count is zero.  To extend to a different $d$-star $\mathcal{S}=(y,x_1,...,x_k,x_{k+1},...,x_d)$, $y$ must be chosen from the set $\mathcal{Q}(B)$.  Let $\l$ be the smallest positive integer satisfying 
$$
|\mathcal{Q}(B)|\leq q^{\l},
$$
so that there are at most $q^{\l}$ choices for $y\in \mathcal{Q}(B)$.  Given such a choice, we count the number of ways to extend the leaf set to obtain a valid star $\mathcal{S}$.  Since 
$$
|\mathcal{Q}(B)|>q^{\l-1},
$$
Lemma \ref{sizelemma} tells us that there exists a subset $J\subseteq \mathcal{Q}(B)$ with $\l$ points, not containing $y$, such that $\{y\}\cup J$ is linearly independent.  For $x\in E\sm B$, let
$$
\Phi_x=J\cap \{z\in E: x\cdot z=t\}.
$$
Suppose that the leaf set of $\mathcal{S}$ is $L=A\cup B$, so that
$$
A=\{x_{k+1},...,x_d\}.
$$
If $B$ is bad in $\mathcal{S}$, then
$$
\bigcup_{i=k+1}^d\Phi_{x_i}=J.
$$
Given some set $Z\subseteq J$, for any $x\in E$ satisfying $\Phi_x=Z$, we see that $x$ lies on the hyperplane $H_z:=\{x: x\cdot z=t\}$ for each $z\in Z$.  Since we already fixed the point $y$ in the star $\mathcal{S}=(y,x_1,...,x_d)$, $x$ also lies on $H_y$.  Since $\{y\}\cup Z$ is linearly independent, this means there are at most $q^{d-1-|Z|}$ choices for $x\in E$ satisfying $\Phi_x=Z$.  Therefore, summing over all possible collections of $d-k$ subsets of $J$ whose union is $J$, we find that the number of stars $\mathcal{S}$ containing $B$ in the leaf set is at most
\begin{align*}
q^{\l}\sum_{\substack{(Z_1,...,Z_{d-k} )\\ \bigcup{Z_i}=J}}\prod_{i=1}^{d-k}q^{d-1-|Z_i|}
&=q^{(d-1)(d-k)+\l}\sum_{\substack{(Z_1,...,Z_{d-k} )\\ \bigcup{Z_i}=J}}\prod_{i=1}^{d-k}q^{-|Z_i|}
\\
&=q^{d^2-kd-d+k+\l}\sum_{\substack{(Z_1,...,Z_{d-k} )\\ \bigcup{Z_i}=J}}q^{-\sum_{i=1}^{d-k}|Z_i|}
\\
&\leq q^{d^2-kd-d+k+\l}\sum_{\substack{(Z_1,...,Z_{d-k} )\\ \bigcup{Z_i}=J}}q^{-\l}
\\
&\leq C_d'q^{d^2-kd-d+k},
\end{align*}
where $C_d'$ is the number of ways to write $J=\bigcup_{i=1}^{d-k}Z_i$.  $C_d'$ depends only on $d$, since $|J|=\l<d$.  
\end{proof}

We are now ready to prove Theorem \ref{main}.

\begin{proof}[Proof of Theorem \ref{main}]
For each $k=1,...,d-1$, let $M_k(E)$ denote the number of $d$-stars in $\mathcal{G}_t(E)$ admitting a bad set of size $k$, and let $M(E)$ denote the total number of $d$-stars admitting a bad set of any size.  If we can show that $M(E)<\frac{|E|^{d+1}}{3q^d}$, then it follows from Lemma \ref{independent} that there exists some $d$-star in $\mathcal{G}_t(E)$ which admits no bad set, and hence the VC-dimension of $\mathcal{H}_t(E)$ is equal to $d$.  Using Lemma \ref{bad}, we see that
$$
M(E)\leq\sum_{k=1}^{d-1}M_k(E)\leq C_d'\sum_{k=1}^{d-1}|E|^kq^{d^2-kd-d+k}
\leq (d-1)C_d'|E|^{d-1}q^{d-1}.
$$
The last step follows from the assumption that $|E|\geq q^{d-1}$, meaning that the summand is largest when $k$ is largest.

Therefore, $M(E)<\frac{|E|^{d+1}}{3q^d}$ whenever
$$
|E|\geq C_dq^{d-\frac{1}{2}}.
$$
We already needed the stronger restriction $|E|\geq q^{d-\frac{1}{d-1}}$ to apply Lemma \ref{independent}, and this completes the proof.
\end{proof}

\section{Future work}
One possible direction of future work is to resolve the problem of whether there exists $\alpha<d$ such that whenever $|E|\geq C_dq^{\alpha}$ for some constant $C_d$ independent of $q$, 
$$
\text{VCdim}(\mathcal{H}_t^{dist}(E))=d+1.
$$ This problem is still open for $d>2$, and the techniques used in this paper do not directly apply in that setting. The difficulty comes from the fact that we would need to solve the same graph embedding problem in a lower dimensional space, and generally these problems are easier in higher dimensions. 

Another direction for future research is to consider other classifiers and see whether these techniques or others apply to computing the VC-dimension of those classifiers restricted to large subsets of $\mathbb F_q^d$.

\end{document}